\documentclass[notitlepage,11pt,reqno]{amsart}
\usepackage[foot]{amsaddr}
\usepackage[numbers,sort]{natbib}

\usepackage{amssymb,nicefrac,bm,upgreek,mathtools,verbatim,esint}
\usepackage[final]{hyperref}
\usepackage[mathscr]{eucal}
\usepackage{dsfont}

\usepackage{color}
\definecolor{dmagenta}{rgb}{.4,.1,.5}
\definecolor{dblue}{rgb}{.0,.0,.5}
\definecolor{mblue}{rgb}{.0,.0,.8}
\definecolor{ddblue}{rgb}{.0,.0,.4}
\definecolor{dred}{rgb}{.6,.0,.0}
\definecolor{dgreen}{rgb}{.0,.5,.0}
\definecolor{Eeom}{rgb}{.0,.0,.5}

\usepackage[normalem]{ulem}

\usepackage[margin=1in]{geometry}
\allowdisplaybreaks

\newtheorem{lemma}{Lemma}[section]
\newtheorem{theorem}{Theorem}[section]

\theoremstyle{definition}

\newtheorem{assumption}{Assumption}[section]

\theoremstyle{remark}

\newtheorem{remark}{Remark}[section]

\numberwithin{equation}{section}

\hypersetup{
  colorlinks=true,
  citecolor=dblue,
  linkcolor=dblue,
  frenchlinks=false,
  pdfborder={0 0 0},
  naturalnames=false,
  hypertexnames=false,
  breaklinks}
\usepackage[capitalize,nameinlink]{cleveref}

\crefname{section}{Section}{Sections}
\crefname{subsection}{Subsection}{Subsections}
\crefname{condition}{Condition}{Conditions}
\crefname{hypothesis}{Hypothesis}{Conditions}
\crefname{assumption}{Assumption}{Assumptions}
\crefname{lemma}{Lemma}{Lemmas}
\crefname{claim}{Claim}{Claims}

\Crefname{figure}{Figure}{Figures}

\crefformat{equation}{\textup{#2(#1)#3}}
\crefrangeformat{equation}{\textup{#3(#1)#4--#5(#2)#6}}
\crefmultiformat{equation}{\textup{#2(#1)#3}}{ and \textup{#2(#1)#3}}
{, \textup{#2(#1)#3}}{, and \textup{#2(#1)#3}}
\crefrangemultiformat{equation}{\textup{#3(#1)#4--#5(#2)#6}}%
{ and \textup{#3(#1)#4--#5(#2)#6}}{, \textup{#3(#1)#4--#5(#2)#6}}%
{, and \textup{#3(#1)#4--#5(#2)#6}}

\Crefformat{equation}{#2Equation~\textup{(#1)}#3}
\Crefrangeformat{equation}{Equations~\textup{#3(#1)#4--#5(#2)#6}}
\Crefmultiformat{equation}{Equations~\textup{#2(#1)#3}}{ and \textup{#2(#1)#3}}
{, \textup{#2(#1)#3}}{, and \textup{#2(#1)#3}}
\Crefrangemultiformat{equation}{Equations~\textup{#3(#1)#4--#5(#2)#6}}%
{ and \textup{#3(#1)#4--#5(#2)#6}}{, \textup{#3(#1)#4--#5(#2)#6}}%
{, and \textup{#3(#1)#4--#5(#2)#6}}

\crefdefaultlabelformat{#2\textup{#1}#3}

\makeatletter
\DeclareRobustCommand\widecheck[1]{{\mathpalette\@widecheck{#1}}}
\def\@widecheck#1#2{%
    \setbox\z@\hbox{\m@th$#1#2$}%
    \setbox\tw@\hbox{\m@th$#1%
       \widehat{%
          \vrule\@width\z@\@height\ht\z@
          \vrule\@height\z@\@width\wd\z@}$}%
    \dp\tw@-\ht\z@
    \@tempdima\ht\z@ \advance\@tempdima2\ht\tw@ \divide\@tempdima\thr@@
    \setbox\tw@\hbox{%
       \raise\@tempdima\hbox{\scalebox{1}[-1]{\lower\@tempdima\box
\tw@}}}%
    {\ooalign{\box\tw@ \cr \box\z@}}}
\makeatother

\DeclareMathOperator{\Exp}{\mathbb{E}} 
\DeclareMathOperator{\Prob}{\mathbb{P}} 
\newcommand{\E}{\mathrm{e}}          
\newcommand{\Rd}{{\mathbb{R}^d}}       
\newcommand{\Ind}{\mathds{1}}            

\newcommand{\uuptau}{\Breve{\uptau}}
\newcommand{\uupsigma}{\Breve{\upsigma}}

\newcommand{\sB}{\mathscr{B}}    
\newcommand{\cC}{\mathcal{C}}     

\newcommand{\sM}{\mathscr{M}}     

\newcommand{\abs}[1]{\lvert#1\rvert}

\DeclareMathOperator{\FLa}{(-\Delta)^{\nicefrac{\alpha}{2}}}
\DeclareMathOperator{\FLb}{(-\Delta)^{\nicefrac{\beta}{2}}}
\DeclareMathOperator{\gFLa}{-(-\Delta)^{\nicefrac{\alpha}{2}}}
\DeclareMathOperator{\gFLb}{-(-\Delta)^{\nicefrac{\beta}{2}}}
\DeclareMathOperator{\dist}{dist}

\begin{document}

\title[Liouville type results for the fractional Laplacian]{Liouville type results for systems of equations involving fractional Laplacian  in exterior domains}

\author{Anup Biswas}
\address{ Department of Mathematics, Indian Institute of Science Education and Research, Dr. Homi Bhabha Road, Pashan, Pune 411008, India}
\email{anup@iiserpune.ac.in}

\begin{abstract}
In this article we present a simple and unified probabilistic approach to prove nonexistence of positive super-solutions for systems of equations
involving potential terms and the fractional Laplacian
in an exterior domain. Such problems arise in the analysis of a priori estimates of solutions. The class of problems we consider in this article is quite general compared to the literature.
The main ingredient for our proofs is the hitting time estimates for the symmetric $\alpha$-stable process and probabilistic representation of the super-solutions.

\end{abstract}



\maketitle

\section{Introduction}

A conjecture of J. Serrin states that for any $p, q>0,$ satisfying
$$\frac{1}{p+1}+\frac{1}{q+1}>\frac{d-2}{d},$$
the Lane-Emden system
\begin{align*}
-\Delta u &= v^p \quad \text{in}\; \Rd,
\\
-\Delta v &= u^q \quad \text{in}\; \Rd,
\end{align*}
does not have any non-trivial non-negative, bounded solution. A complete answer to this conjecture is still unknown. However, there are some partial results available in this direction, see \cite{DF94, M96, SZ98, SZ96}.
Interest in such Liouville type properties arose from a seminal work of Gidas and Spruck \cite{GS81} where it was shown that the equation
$$-\Delta u = u^p \quad \text{in}\; \Rd,$$
has only trivial non-negative solution for $1\leq p<\frac{d+2}{d-2}$. A large amount of works have been done generalizing this result in various directions. To cite a few we refer to 
\cite{DF94, GS14, L04, M96, SZ98, SZ96, S17, CPZ}. Recently, \cite{GS17} has considered the scalar equation with a potential term in an exterior domain and obtained sufficient conditions 
for the validity of Liouville type properties.

In this article we are broadly interested in systems of equations of the form
\begin{equation*}
\left\{
\begin{split}
\FLa u &\geq f(x, u, v) \quad \text{in}\; D,
\\
\FLb v & \geq g(x, u, v) \quad \text{in}\; D,
\\
u, v &\geq 0\quad \text{in}\; \Rd,
\end{split}
\right.
\end{equation*}
where $D$ is either $\Rd$ or an exterior domain and $f, g$ satisfy some structural condition. See Section~\ref{S-main} for exact conditions. There are few works available in the literature 
for these Lane-Emden type systems under the assumption that $f(x, u, v)=v^p, g(x, u, v)=u^q$ and $D=\Rd$ (see \cite{DKK, CL09, QX16, Y13}). 
Furthermore, all these works consider solution instead of 
super-solution. See also \cite{CLO, JLX, WX16} for Liouville type properties of scalar 
semilinear equations involving fractional Laplacian. These problems are a variant of the classical Liouville problem which states that  all bounded $\alpha$-harmonic functions (i.e., the solution of $\FLa u=0$) in $\Rd$ are constants.
We refer to \cite{MMF, FV, RS16} for results in this direction and its extension. Besides their intrinsic interest, Liouville type results are an important tool for proving existence results for related Dirichlet problems
for elliptic equations and systems. See for instance, \cite{DM10, DM18} and references therein for the case of the Laplacian operators and \cite{BDG, BN18} for the case fractional Laplacian.

In this article we propose a simple and unified probabilistic approach which is capable to deal with such problems 
in the exterior domains for a large family of $f, g$. We refer the readers to the discussion at the end of Section~\ref{S-main} to compare the sharpness of
our results to the existing literature.
Recently, a similar probabilistic approach has been used in \cite{BL18} for studying Liouville type properties for local Dirichlet forms
on metric measure spaces.


\section{Main results}\label{S-main}
We consider the following systems of equations
\begin{equation}\tag{A1}\label{A1}
\left\{
\begin{split}
\FLa u &\geq f(x, v) \quad \text{in}\; B^c,
\\
\FLb v & \geq g(x, u) \quad \text{in}\; B^c,
\\
u, v &\geq 0\quad \text{in}\; \Rd.
\end{split}
\right.
\end{equation}
Here $B=\overline{B(0, 1)}$ and $f, g$ are suitable functions satisfying the following hypothesis. 
\begin{assumption}\label{Ass-1}
$f, g:\Rd\times [0, \infty)\to [0, \infty)$ are continuous and $f(\cdot, 0)=g(\cdot, 0)=0$. For every fixed $x$ we have that $f(x, \cdot), \, g(x, \cdot)$ are non-decreasing and furthermore,
for some $p, q>0$ and $U, V:\Rd\to (0, \infty)$ we have
\begin{equation}\label{A2.1}
\liminf_{t\to 0+}\inf_{x\in\Rd} \frac{f(x, t)}{U(x) t^p}>0, \quad \liminf_{t\to 0+}\inf_{x\in\Rd} \frac{g(x, t)}{V(x) t^q}>0\,.
\end{equation}
\end{assumption}
Also, on the whole Euclidean space $\Rd$ we consider non-negative solutions of 
\begin{equation}\tag{B1}\label{B1}
\left\{
\begin{split}
\FLa u &\geq f(x, u, v) \quad \text{in}\; \Rd,
\\
\FLb v & \geq g(x, u, v) \quad \text{in}\; \Rd,
\\
u, v &\geq 0\quad \text{in}\; \Rd,
\end{split}
\right.
\end{equation}
where $f, g$ satisfy the following
\begin{assumption}\label{Ass-2}
$f, g:\Rd\times [0, \infty)\times[0, \infty)\to [0, \infty)$ are continuous and $f(\cdot, 0, 0)=g(\cdot, 0, 0)=0$. For every fixed $x$ we have that $f(x, \cdot, \cdot), \, g(x, \cdot, \cdot)$ are component-wise non-decreasing and furthermore,
for some $ p_2, q_1\geq 0, p_1, q_2\geq 1$, and $U, V:\Rd\to (0, \infty)$ we have
\begin{equation}\label{A2.2}
\liminf_{(t,s)\to (0+, 0+)}\inf_{x\in\Rd} \frac{f(x, t, s)}{U(x) t^{p_1} s^{p_2}}>0, \quad \liminf_{(t,s)\to (0+, 0+)}\inf_{x\in\Rd} \frac{g(x, t)}{V(x) t^{q_1} s^{q_2}}>0\,.
\end{equation}
\end{assumption}

Let us define 
$$\Phi_U(r)=\inf_{\frac{r}{2}\leq\abs{x}\leq \frac{3r}{2}}\int_{B(x, \frac{r}{4})} U(y) dy,
 \quad \Phi_V(r)=\inf_{\frac{r}{2}\leq\abs{x}\leq \frac{3r}{2}}\int_{B(x,  \frac{r}{4})} V(y) dy,$$
 where $B(x, r)$ denotes the ball of radius $r$ centered at $x$.
Our first main result is as follows.
\begin{theorem}\label{T2.1}
Grant Assumption~\ref{Ass-1}.
Suppose that $\alpha, \beta\in (0, 2\wedge d)$, and
\begin{equation}\label{ET2.1A}
\lim_{r\to\infty}\; \max\left\{\frac{r^{d-\alpha}}{\Phi_U(r)}, \frac{r^{d-\beta}}{\Phi_V(r)} \right\}=0,
\end{equation}
and one of the following hold
\begin{align}
\lim_{r\to\infty}\frac{1}{r^{(q+1 )d-\beta-\alpha q}} (\Phi_U(r))^{\frac{1}{p}} \Phi_V(r) &=\infty,\label{ET2.1A1}
\\
\lim_{r\to\infty}\frac{1}{r^{(p+1 )d-\beta p-\alpha}} \Phi_U(r) (\Phi_V(r))^{\frac{1}{q}} &=\infty.\label{ET2.1A2}
\end{align}
Then \eqref{A1} does not have any solution other than $u=v=0$. In particular, if we have $pq\leq1$ and $0<c_1\leq \min\{U(x), V(x)\}$, for some
constant $c_1$ and all $x\in B^c$,
 then \eqref{A1} has only trivial solutions.
\end{theorem}

Our second main result concerns \eqref{B1}.
\begin{theorem}\label{T2.2}
Grant Assumption~\ref{Ass-2}. Suppose that $\alpha, \beta\in (0, 2\wedge d)$, and
\begin{equation}\label{ET2.2A}
\lim_{r\to\infty}\; \max\left\{\frac{r^{d-\alpha}}{\Phi^{\frac{1}{p_1}}_U(r)}, \frac{r^{d-\beta}}{\Phi^{\frac{1}{q_2}}_V(r)},
\frac{r^{(p_1+p_2+q_1+q_2)d}}{r^{\alpha(p_1+q_1)} r^{\beta(p_2+q_2)}}\frac{1}{\Phi_U(r)\Phi_V(r)} \right\}=0.
\end{equation}
Then \eqref{B1} has only trivial super-solutions i.e. either $u=0$ or $v=0$ in $\Rd$.
\end{theorem}

Theorem~\ref{T2.2} can be further improved for a particular choice of $f$ and $g$ as follows.
\begin{theorem}\label{T2.3}
Suppose that $\alpha, \beta\in (0, 2\wedge d)$. Consider the problem
\begin{equation}\tag{C1}\label{C1}
\left\{
\begin{split}
\FLa u &\geq U(x) u^{p_1} v^{p_2} \quad \text{in}\; \Rd,
\\
\FLb v & \geq V(x) u^{q_1} v^{q_2} \quad \text{in}\; \Rd,
\\
u, v &\geq 0\quad \text{in}\; \Rd,
\end{split}
\right.
\end{equation}
where $p_1, q_2>0,\, p_2, q_1\geq 0$. Then \eqref{C1} has only trivial solutions if one of the following holds.
\begin{enumerate}
\item[(i)] $(p_1+q_1)\wedge(p_2+q_2)\geq 1$ and 
\begin{equation}\label{ET2.3A}
\liminf_{r\to\infty}\; \frac{r^{(p_1+p_2+q_1+q_2)d}}{r^{\alpha(p_1+q_1)} r^{\beta(p_2+q_2)}}\frac{1}{\Phi_U(r)\Phi_V(r)}=0.
\end{equation}
\item[(ii)] $q_2<1$ and one of the following holds.
\begin{itemize}
\item[(a)] $(p_1-1)(1-q_2)+p_2q_1< 0$ and
\begin{equation}\label{ET2.3A0}
\limsup_{r\to\infty}\; \frac{\Phi^{1-q_2}_U(r)}{r^{(d-\alpha)(1-q_2)}} \frac{\Phi^{p_2}_V(r)}{r^{(d-\beta)p_2}} =\infty.
\end{equation}
\item[(b)] $(p_1-1)(1-q_2)+p_2q_1\geq 0$, and
\begin{equation}\label{ET2.3B}
\limsup_{r\to\infty}\; \frac{\Phi^{p_2}_V(r)\Phi^{1-q_2}_U(r)}{r^{(d-\beta) p_2}} \frac{1}{r^{(d-\alpha)(p_2 q_1 + p_1 (1-q_2))}}=\infty.
\end{equation}
\end{itemize}
\item[(iii)] $p_1<1$ and one of the following holds.
\begin{itemize}
 \item[(a)] $(1-p_1)(q_2-1)+p_2q_1< 0$ and 
 \begin{equation}\label{ET2.3A00}
\limsup_{r\to\infty}\; \frac{\Phi^{q_1}_U(r)}{r^{(d-\alpha)q_1}} \frac{\Phi^{(1-p_1)}_V(r)}{r^{(d-\beta)(1-p_1)}}  =\infty.
\end{equation}
\item[(b)] $(1-p_1)(q_2-1)+p_2q_1\geq 0$ and
\begin{equation}\label{ET2.3C}
\limsup_{r\to\infty}\; \frac{\Phi^{q_1}_U(r)\Phi^{1-p_1}_V(r)}{r^{(d-\alpha) q_1}} \frac{1}{r^{(d-\beta)(p_2 q_1 + q_2 (1-p_1))}}=\infty.
\end{equation}
\end{itemize}

\item[(iv)] $(p_1+q_1)\vee(p_2+q_2)\leq 1$ and 
\begin{equation}\label{ET2.3C1}
\liminf_{r\to\infty}\; \frac{r^{d-\alpha}}{\Phi_U(r)}\cdot \frac{r^{d-\beta}}{\Phi_V(r)}=0.
\end{equation}
\end{enumerate}
\end{theorem}

\begin{remark}
Suppose that $p_1+p_2=q_1+q_2=\eta\geq 1$. Then we note that
\begin{align*}
(p_1-1)(1-q_2)+p_2q_1= (p_1-1)(1-q_2)+(\eta-p_1)(\eta-q_2) &= (\eta^2 -1) -(\eta-1)(p_1+q_2)
\\
&= (\eta-1) (\eta +1 -p_1 -q_2)\geq 0,
\end{align*}
if $q_2\leq 1$. Similarly, the condition in Theorem~\ref{T2.3}(iii)(b) becomes redundant in this case. The condition 
$p_1+p_2=q_1+q_2=\eta> 1$ has been used in \cite{S17} to study a similar problem for the Laplacian operator. 
\end{remark}
Problems similar to \eqref{B1}-\eqref{C1} have been studied by D'Ambrosio and Mitidieri in $\Rd$ but for quasilinear operators \cite{DM13, DM18}. 
See also \cite{DM10} for results related to
scalar quasi-linear equations.
In \cite{DM14}, the same authors consider
a more general systems of nonlocal equations in $\Rd$ with Hardy type weights and establish similar Liouville type results.

Finally, we extend the above results in the exterior domain.
\begin{theorem}\label{T2.4}
Suppose that $\alpha, \beta\in (0, 2\wedge d)$. Consider the problem
\begin{equation}\tag{D1}\label{D1}
\left\{
\begin{split}
\FLa u &\geq U(x) u^{p_1} v^{q_1} \quad \text{in}\; B^c,
\\
\FLb v & \geq V(x) u^{p_2} v^{q_2} \quad \text{in}\; B^c,
\\
u, v &\geq 0\quad \text{in}\; \Rd,
\end{split}
\right.
\end{equation}
where $p_1, q_2\geq 0$, $p_2, q_1>0$, and
\begin{equation}\label{ET2.4A}
\lim_{r\to\infty}\; \max\left\{\frac{r^{d-\alpha}}{\Phi_U(r)}, \frac{r^{d-\beta}}{\Phi_V(r)} \right\}=0.
\end{equation}
Then \eqref{D1} has only trivial solution if the following hold.
\begin{enumerate}
\item[(i)] If $p_2 q_1+p_1< 1$, we have
\begin{equation}\label{ET2.4A11}
\liminf_{r\to\infty}\left[\frac{\Phi_V(r)}{r^{(d-\beta)(1+q_2)}}\right]^{p_2} \frac{\Phi_U(r)}{r^{d-\alpha}}>0.
\end{equation}
\item[(ii)] If $p_2q_1+q_2< 1$, we have
\begin{equation}\label{ET2.4A12}
\liminf_{r\to\infty}\left[\frac{\Phi_U(r_n)}{r_n^{(d-\alpha)(1+p_1)}}\right]^{q_1} \frac{\Phi_V(r_n)}{r_n^{d-\beta}}>0.
\end{equation}
\item[(iii)] if $p_2 q_1 + p_1\geq 1$, we have 
\begin{equation}\label{ET2.4A1}
\lim_{r\to\infty}\; \frac{1}{r^{(d-\alpha)(p_2 q_1+p_1)}}\frac{1}{r^{(d-\beta)(p_2q_2+p_2)}} \Phi^{p_2}_V(r) \Phi_U(r)=\infty.
\end{equation}
\item[(iv)] If $p_2q_1+q_2 \geq 1$, we have
\begin{equation}\label{ET2.4A2}
\lim_{r\to\infty}\; \frac{1}{r^{(d-\beta)(p_2q_1+q_2)}}\frac{1}{r^{(d-\alpha)(p_2q_1+q_1)}} \Phi_V(r) \Phi^{q_1}_U(r)=\infty.
\end{equation}
\end{enumerate}
\end{theorem}

By a super-solution we mean classical super-solution i.e.  
$u\in \cC^{\alpha+}(B^c)\cap \cC(\Rd)\cap L^1(\Rd, \omega_\alpha)$ and $v\in\cC^{\beta+}(B^c)\cap \cC(\Rd)\cap L^1(\Rd, \omega_\beta)$. Here
$\cC^{\alpha+}(B^c)$ is the collection of functions with the following property: for every $f\in \cC^{\alpha+}(B^c)$ and any compact set $K\Subset B^c$ there exists
$\gamma>0$ such that $f\in\cC^{\alpha+\gamma}(K)$. $L^1(\Rd, \omega_\alpha)$ denotes the class of integrable functions with respect to the weight function
$\omega_\alpha(x)=\frac{1}{1+\abs{x}^{d+\alpha}}$.
All our main results can be improved to viscosity super-solutions. As we see in the proofs below that additional regularity is used to find a stochastic representation of the 
super-solutions (see Lemma~\ref{DL3.1}) and this can also be obtained from the comparison principle (cf. \cite{CS09}) which only requires the continuity of the solutions.
 See Theorem~\ref{T-viscosity} below for more
details.

\subsection{Comparison and discussion}
Before we proceed to the proofs let us compare these results to those available in the literature.
\begin{enumerate}
\item[(a)] \cite{QX16} studies Liouville property of  the solutions of \eqref{A1} for $f(x, t)=t^p, g(x, t)=t^q$ and $\alpha=\beta$ whereas \cite{DKK}
establishes the Liouville property for weak solutions with $f(x, t)=t^p, g(x, t)=t^q$. Both of these works consider the equations on $\Rd$.
Choosing $U=V=1$ in Theorem~\ref{T2.1} we see that \eqref{ET2.1A} holds, and \eqref{ET2.1A1}-\eqref{ET2.1A2} are equivalent to the condition
$$d< \max\left\{\frac{(\beta + \alpha q) p}{pq-1}, \frac{(\beta p + \alpha) q}{pq-1}\right\}, \quad pq>1. $$
Note the above condition is same as in \cite{DKK}. The techniques in \cite{DKK,QX16} crucially use the special form of $f$ and $g$ and may not
be useful for general $f, g,$ like ours and also do not work for super-solutions.

\item[(b)] Suppose that $U(x)\gtrsim \abs{x}^m, m>-\alpha$, and $V(x)\gtrsim \abs{x}^{n}, n>-\beta$. Such potential functions are considered in
\cite{BL18, CDM08a, CPZ, GS17}. It is easy to see that we have $\Phi_U(r)\gtrsim r^{d+m}$ and $\Phi_V(r)\gtrsim r^{d+n}$. Therefore, \eqref{ET2.1A} holds.
Again, \eqref{ET2.1A1}-\eqref{ET2.1A2} can be replaced by 
$$d< \max\left\{\frac{m+np+(\beta + \alpha q) p}{pq-1}, \frac{mq+n+(\beta p + \alpha) q}{pq-1}\right\}, \quad pq>1.$$
This generalizes \cite{GS17} to Lane-Emden type systems for the fractional Laplacian. We can also deduce similar condition for other results as well.
For instance, \eqref{ET2.3A} can be replaced by
$$d<\frac{\alpha(p_2+q_2) +\beta(p_1+q_1) + m+ n}{p_1+p_2+q_1+q_2-2},$$
or \eqref{ET2.4A1} can be replaced by
$$d<\frac{p_2 n + m + \alpha(p_2q_1+p_1) + \beta(p_2 q_2 + p_2)}{p_2q_1+p_1+p_2 q_2-1}.$$

\item[(c)] In a recent work, Sun \cite{S17} (see also \cite{Y13}) has studied the Liouville property of \eqref{C1} for the Laplacian for $U=V=1$ and $p_1+p_2=q_1+q_2=\eta>1$.
To compare our result, let us assume that $\alpha=\beta$. Then it is easily seen that $d<\frac{\alpha \eta}{\eta-1}$ implies \eqref{ET2.3A}. Now we
simplify \eqref{ET2.3B} as well. We calculate
\begin{align*}
(d-\alpha)(p_2 + p_2 q_1 + p_1(1-q_2)) &= (d-\alpha)(\eta-p_1 + (\eta-p_1) (\eta-q_2) + p_1(1-q_2))
\\
&= (d-\alpha)(\eta + \eta^2 -p_1\eta-q_2\eta)
\\
&= (d-\alpha)\eta(p_2+1-q_2)\,.
\end{align*}
Then \eqref{ET2.3B} is equivalent to $d<\frac{\alpha \eta}{\eta-1}$. Likewise, we reach at the same conclusion for \eqref{ET2.3C}. Note that
\cite{S17} also obtains the same critical value $\frac{\alpha \eta}{\eta-1}$ for $\alpha=2$.

D'Ambrosio and Mitidieri also study the Liouville  property of \eqref{C1} for the Laplacian operator with $U=1=V$. We compare our results to \cite[Theorem~1]{DM18}. Let $p_1, q_2\in (0, 1)$, $p_2 q_1>0$ and $\alpha=\beta$. We note
that \eqref{ET2.3A0} and \eqref{ET2.3A00} hold in this case since $\Phi_V(r)\simeq r^d\simeq \Phi_U(r)$ for $r\geq 1$. An easy computation shows that
if we have
$$d \left(1 - \frac{(1-p_1)(1-q_2)}{p_2 q_1}\right)< \max\left\{\alpha + \alpha \frac{p_2+(1-q_2)p_1}{p_2 q_1}, \alpha + \alpha \frac{q_1+(1-p_1)q_2}{p_2 q_1}\right\}$$
then \eqref{ET2.3B} and \eqref{ET2.3C} hold. Putting $\alpha=2$ we see that the above bound coincides with the one in \cite[Theorem~1]{DM18}.

\item[(d)] Recently, Liouville type equations with potential have also been considered in \cite{CDM08a, CDM08b, MP99, MP04} where an integral condition
is proposed on the potential.
Inspired by these works we can also
have a similar condition imposed on potential $U$ and $V$ to guarantee the nonexistence of solutions. Denote by
\begin{align*}
r^{-\frac{\beta p}{p-1}}\int_{B(0, \frac{3r}{2})\setminus B(0, r/2)} U^{\frac{-1}{p-1}}(y) dy &=\ell_U(r),
\\
r^{-\frac{\alpha q}{q-1}}\int_{B(0, \frac{3r}{2})\setminus B(0, r/2)} V^{\frac{-1}{q-1}}(y) dy &=\ell_V(r). 
\end{align*}
Suppose that one of the followings hold: for $p, q>1$,
\begin{align}
\lim_{r\to\infty} \ell^{\frac{p-1}{p}}_U(r) \ell^{q-1}_V(r)&= 0, \label{E2.12}
\\
\lim_{r\to\infty} \ell^{p-1}_U(r) \ell^{\frac{q-1}{q}}_V(r) &=0. \label{E2.13}
\end{align}
Let us now show that \eqref{E2.12} (\eqref{E2.13}) implies \eqref{ET2.1A1}(\eqref{ET2.1A2}, respectively).
First of all we note that
$$\liminf_{r\to\infty} \inf_{\frac{r}{2}\leq \abs{x}\leq \frac{3r}{2}}\; \frac{|(B(0, \frac{3r}{2})\setminus B(0, \frac{r}{2}))\cap B(x, \frac{r}{4})|}{r^d}>0.$$
Denote by $A(x, r)= (B(0, \frac{3r}{2})\setminus B(0, \frac{r}{2}))\cap B(x, \frac{r}{4})$ . 
Then for any $\frac{r}{2}\leq \abs{x}\leq \frac{3r}{2}$ we get that
\begin{align*}
\int_{B(x, \frac{r}{4})} U(y) dy \geq \int_{A(x, r)} U(y) dy &\geq |A(x, r)| \left[\fint_{A(x, r)} U^{\frac{-1}{p-1}}(y) dy\right]^{-(p-1)}
\\
&\geq |A(x, r)|^p \left[\int_{B(0, \frac{3r}{2})\setminus B(0, \frac{r}{2})} U^{\frac{-1}{p-1}}(y) dy\right]^{-(p-1)}
\\
&\gtrsim r^{d p}  r^{-\beta p} \frac{1}{\ell^{p-1}_U(r)},
\end{align*}
where in the first line we used Jensen's inequality. Thus 
$$\Phi^{\frac{1}{p}}_U(r)\gtrsim r^{d-\beta}\frac{1}{\ell^{\frac{p-1}{p}}_U(r)}.$$
Similarly, 
$$\Phi^{\frac{1}{q}}_V(r)\gtrsim r^{d-\alpha}\frac{1}{\ell^{\frac{q-1}{q}}_V(r)}.$$
Hence
\begin{align*}
\frac{1}{r^{(q+1 )d-\beta-\alpha q}} (\Phi_U(r))^{\frac{1}{p}} \Phi_V(r) &\gtrsim \frac{1}{\ell^{\frac{p-1}{p}}_U(r) \ell_V^{q-1}(r)},
\\
\frac{1}{r^{(p+1 )d-\beta p-\alpha}} \Phi_U(r) (\Phi_V(r))^{\frac{1}{q}} &\gtrsim \frac{1}{\ell_U^{p-1}(r) \ell^{\frac{q-1}{q}}_V(r)}.
\end{align*}
Therefore, \eqref{E2.12} (\eqref{E2.13}) implies \eqref{ET2.1A1}(\eqref{ET2.1A2}, respectively).

\item[(e)] Suppose that $\alpha=\beta$ and $p_1+p_2=q_1+q_2=\eta>1$ in Theorem~\ref{T2.4}.  Denote by
\begin{align*}
r^{-\frac{\alpha \eta}{\eta-1}}\int_{B(0, \frac{3r}{2})\setminus B(0, r/2)} U^{\frac{-1}{\eta-1}}(y) dy &=\ell_U(r),
\\
r^{-\frac{\alpha \eta}{\eta-1}}\int_{B(0, \frac{3r}{2})\setminus B(0, r/2)} V^{\frac{-1}{\eta-1}}(y) dy &=\ell_V(r). 
\end{align*}
Then the calculations in (d) gives us
$$\Phi_U(r)\gtrsim r^{(d-\alpha)\eta}\frac{1}{\ell^{\eta-1}_U(r)}, \quad
\text{and}\quad \Phi_V(r)\gtrsim r^{(d-\alpha)\eta}\frac{1}{\ell^{\eta-1}_V(r)}.$$
Thus if we assume that
$$\lim_{r\to\infty}\; \max\left\{\ell^{\eta-1}_U(r)\ell^{p_2(\eta-1)}_V(r), \ell^{q_1(\eta-1)}_U(r)\ell^{\eta-1}_V(r)\right\}=0,$$
then \eqref{ET2.4A1}--\eqref{ET2.4A2} holds. In particular, for $U=1=V$,
\eqref{ET2.4A1}--\eqref{ET2.4A2} holds if $d<\frac{\alpha \eta}{\eta-1}$ and in addition, if $\max\{p_1, q_2\}\leq \frac{\alpha}{d-\alpha}$
holds then we also have \eqref{ET2.4A11}-\eqref{ET2.4A12}.
\end{enumerate}

In all our above results we can replace $B$ by any compact set $K$ as long as $K^c$ remains connected. In this article we use the notation $\kappa, \kappa_1, \ldots$ for non-specific constants whose value might change
from line to line.

\section{Proofs}\label{S-proof}
This section is devoted to the proofs of our main results stated in the previous section. We start by recalling few facts about symmetric stable processes.
Let $X$ be the $d$-dimensional, spherically symmetric $\alpha$-stable process with $\alpha\in (0, 2)$ defined on a complete probability space 
$(\Omega, \mathcal{F}, \Prob)$. In particular, 
$$\Exp\Bigl[\E^{i\xi\cdot (X_t-X_0)}\Bigr]\;=\; \E^{-t\abs{\xi}^\alpha}, \quad \text{for every}\; \xi\in\Rd, \; t\geq 0\,.$$
When $X_0=x$ we say that the $\alpha$-stable process $X$ starts at $x$.
The generator of $X$ is given by $\gFLa$. We use the notation $Y$ to denote the $\beta$-stable process, defined over the same 
probability space, and its generator is denoted by
$\gFLb$ (cf. \cite{Bog-book} for more details). 
Given any set $A\subset\Rd$, we shall denote by $\uptau_A$ the exit time of $X$ from $A$ i.e., 
$$\uptau_A=\inf\{t>0\; : \; X_t\notin A\},$$
and we denote by $\uuptau_A=\uptau_{A^c}$ i.e., $\uuptau_A$ denotes the hitting time to $A$. Likewise, we use $\upsigma_A$ to denoted the exit time of $Y$ and $\uupsigma_A$ would be used
to denote the hitting time of $Y$ to $A$. In the proofs below we use Dynkin's formula for the functions in
$\cC^{\alpha+}(B^c)\cap\cC(\Rd)\cap L^1(\Rd, \omega_\alpha)$. We include a proof below for convenience.

\begin{lemma}\label{DL3.1}
Suppose that $\varphi\in\cC^{\alpha+}(D)\cap\cC(\Rd)\cap L^1(\Rd, \omega_\alpha)$ for some open set $D$. Then for any $\cC^{1,1}$ open set $D_1\Subset D$ we have 
\begin{equation}\label{Dynkin}
\varphi(x)=\Exp_x[\varphi(X_{t\wedge\uptau_{D_1}})] + \Exp_x\left[\int_0^{t\wedge\uptau_{D_1}} \FLa\varphi(X_s) ds\right], \quad \text{for all}\; t\geq 0, \; x\in D_1.
\end{equation}
where $\uptau_{D_1}$ denotes the exit time of $X$ from $D_1$. Moreover, if $\varphi$ is non-negative and $\FLa\varphi\geq \ell$ in $D_1$, for some bounded function $\ell$, we have
\begin{equation}\label{Dynkin1}
\varphi(x)\geq \Exp_x[\varphi(X_{\uptau_{D_1}})] + \Exp_x\left[\int_0^{\uptau_{D_1}}\ell(X_s) ds\right], \quad \text{for all} \; x\in D_1.
\end{equation}
\end{lemma}

\begin{proof}
Since $D_1\Subset D$ we can find another open set $D_2$ such that $D_1\Subset D_2\Subset D$ and $\varphi\in \cC^{\alpha+\gamma}(\bar D_2)$ for some $\gamma>0$.
Now consider a collection of bounded functions $\varphi_n\in\cC^2(D_2)\cap\cC(\Rd)$ satisfying
\begin{equation}\label{DL3.1A}
\sup_{x\in D_1}\abs{\FLa \varphi(x)-\FLa \varphi_n(x)}\to 0, \quad \text{and}\quad \int_{\Rd}\abs{\varphi_n(x)-\varphi(x)}\omega_\alpha(x) dx\to 0, \quad \text{as}\; n\to\infty.
\end{equation}
We can even restrict ourselves to the situation where $\varphi_n-\varphi$ goes to $0$ uniformly over every compact subset of $\Rd$.
This is possible since $\varphi\in \cC^{\alpha+\gamma}(\bar D_2)\cap\cC(\Rd)$, and such a sequence can be constructed by mollifying $\varphi$ and the modifying it outside a large compact set.
Now by It\^{o}'s formula (cf. \cite[Chapter~II.7]{Pro-book}) we get that 
\begin{equation}\label{DL3.1B}
\varphi_n(x)=\Exp_x[\varphi_n(X_{t\wedge\uptau_{D_1}})] + \Exp_x\left[\int_0^{t\wedge\uptau_{D_1}} \FLa\varphi_n(X_s) ds\right], \quad \text{for all}\; t\geq 0, \; x\in D_1.
\end{equation}
Note that using \eqref{DL3.1A} we can pass to the limit in the right most term of \eqref{DL3.1B}. Again, by \cite[Theorem~1.5]{CS05} the
Poisson kernel (i.e. the distribution kernel of $X_{\uptau_{D_1}}$) behaves like $\omega_\alpha$ near infinity.
Thus, using \eqref{DL3.1A},  we can pass to the limit  in the first term  on the RHS of \eqref{DL3.1B} . Hence we obtain \eqref{Dynkin}.

\eqref{Dynkin1} follows by letting $t\to\infty$ in \eqref{Dynkin}, and applying Fatou's lemma in the second term and dominated convergence theorem in the right most term.
\end{proof}

By $G_{\alpha, \sB}$ we denote the Green function of $\FLa$ in the ball $\sB$. This function is uniquely 
characterized by the property
$$\Exp_x\left[\int_0^{\uptau_\sB} f(X_s) ds\right]=\int_\sB G_{\alpha, \sB}(x, y) f(y) dy, \quad \text{for all bounded Borel measurable function}\; f.$$
It is also known that $G_{\alpha, \sB}$ is continuous in $\sB\times \sB$ outside the diagonal \cite[p.~467]{CS05}. We need the following estimate from \cite[Lemma~2.2 and ~6.6]{CS05}. We note that \cite{CS05} considers $d\geq 2$ but the  proofs  of \cite[Lemma~2.2 and ~6.6]{CS05} go through for $\alpha<d\wedge 2$.
\begin{lemma}\label{L2.1}
For some universal constant $C=C(d, \alpha)$ it holds that 
$$\min\left\{\frac{1}{\abs{x-y}^{d-\alpha}}, \frac{\delta^{\nicefrac{\alpha}{2}}_\sB(x) \delta^{\nicefrac{\alpha}{2}}_\sB(y)}{\abs{x-y}^{d}}\right\}C^{-1}\;\leq G_{\alpha, \sB}(x, y)
\leq C\,  \frac{\delta^{\nicefrac{\alpha}{2}}_\sB(x) \delta^{\nicefrac{\alpha}{2}}_\sB(y)}{\abs{x-y}^{d}},$$
for all $x, y\in \sB$ and for any ball $\sB$ in $\Rd$. Here $\delta_\sB(x)=\dist(x, \sB^c)$.
\end{lemma}

Consider a non-negative super-solution of $\FLa \varphi\geq 0$ in $B^c$. Suppose that $\varphi>0$ on $\partial B$ .
 We define
\begin{equation}\label{defi}
\sM_\varphi(r)=\inf_{x\in B(0, r)\cap B^c} \varphi(x).
\end{equation}
\begin{lemma}\label{L2.2}
Suppose that $\alpha<d\wedge 2$ and $\varphi$ is positive in $B^c$. Then the following holds
\begin{itemize}
\item[(a)] For some constant $\kappa_2$ we have for $r\geq 4$ that
$$\sM_\varphi(r)\geq \kappa_2 \frac{1}{r^{d-\alpha}}\,.$$

\item[(b)] There exists a constant $\kappa_3$ satisfying
$$\sM_\varphi(r)\leq \kappa_3 \sM_\varphi(2r), \quad r\geq 4.$$
\end{itemize}
\end{lemma}

\begin{proof}
We consider (a) first. Pick a point $x\in B^c$ so that $\abs{x}=r\geq 4$. Now consider the ball $\sB=B(x, 2r)$ and define $A=\overline{B(0, 2)\cap B^c}$.
Then by \cite[Proposition~7]{G14} (see expression (26) there) we have
\begin{equation}\label{EL2.2A}
\Prob_x(\uuptau_A< \uptau_\sB)\geq \inf_{z\in A} G_{\sB}(x, z) \mathrm{Cap}(A),
\end{equation}
where ``Cap" denotes the capacity function with respect to the process $X$. Again, by \cite[Corollary~3]{G14} we have for some constant $\kappa=\kappa(d)$
that
\begin{equation}\label{EL2.2A0}
\mathrm{Cap}(A)\geq \kappa |A|^{\frac{d-\alpha}{d}},
\end{equation}
for any non-empty Borel set. Using Lemma~\ref{L2.1} we find that for some constant $\kappa_1$, independent of $r$, we have
\begin{equation*}
\inf_{z\in A} G_{\sB}(x, z)\geq \kappa_1 \frac{1}{r^{d-\alpha}}.
\end{equation*}
Plugging these estimates in \eqref{EL2.2A} we obatin
\begin{equation}\label{EL2.2B}
\Prob_x(\uuptau_A< \uptau_\sB)\geq \kappa_2 \frac{1}{r^{d-\alpha}},
\end{equation}
for some constant $\kappa_2$. Now applying Dynkin's formula \eqref{Dynkin} to $\varphi$ we see that
$$\varphi(x)\geq \Exp_x[\varphi(X_{t\wedge\uuptau_{B(0, 2)}\wedge\uptau_\sB})], \quad t\geq 0.$$
Letting $t\to\infty$ and using Fatou's lemma we arrive at
$$\varphi(x)\geq \Exp_x[\varphi(X_{\uuptau_{B(0, 2)}\wedge\uptau_\sB})]\geq [\min_{z\in A} \varphi(z)]\Prob_x(\uuptau_A< \uptau_\sB)
\geq [\min_{z\in A}\varphi(z)] \kappa_2 \frac{1}{r^{d-\alpha}},$$
by \eqref{EL2.2B}. Thus for any $4\leq \abs{x}\leq r$ we obtain
$$\varphi(x)\geq \kappa_3 \frac{1}{\abs{x}^{d-\alpha}}\geq \kappa_3 \frac{1}{r^{d-\alpha}},$$
for some $\kappa_3>0$. Since $\varphi>0$ the result follows.

Now we come to (b). Pick any point $x\in \overline{B(0, 2r)}\setminus \overline{B(0, r)}$ and fix $\sB=B(x, 4r)$ and $A=\overline{B(0, r)\cap B^c}$.
As before, applying \cite[Proposition~7]{G14} we have
\begin{equation}\label{EL2.2C}
\Prob_x(\uuptau_A< \uptau_\sB)\geq \inf_{z\in A} G_{\sB}(x, z) \mathrm{Cap}(A).
\end{equation}
By \eqref{EL2.2A0} we have $\mathrm{Cap}(A)\geq \kappa_4\, r^{d-\alpha}$ for some positive constant $\kappa_4$, independent of $r$. Moreover,
using Lemma~\ref{L2.1} we can find a constant $\kappa_5$ satisfying
\begin{align*}
\inf_{z\in A} G_{\sB}(x, z) &\geq \inf_{z\in A}\,C^{-1}\,\min\left\{\frac{1}{\abs{x-z}^{d-\alpha}}, \frac{\delta^{\nicefrac{\alpha}{2}}_\sB(x) \delta^{\nicefrac{\alpha}{2}}_\sB(z)}{\abs{x-z}^{d}}\right\}
\\
&\geq \inf_{z\in A}\,C^{-1}\,\min\left\{\frac{1}{\abs{4r}^{d-\alpha}}, \frac{(4r)^{\nicefrac{\alpha}{2}} \delta^{\nicefrac{\alpha}{2}}_\sB(z)}{\abs{4r}^{d}}\right\}
\\
&\geq \kappa_5 \,\frac{1}{r^{d-\alpha}}\,.
\end{align*}
Putting these estimates in \eqref{EL2.2C} we find
$$\Prob_x(\uuptau_A< \uptau_\sB)\geq\kappa_4 \kappa_5.$$
As before, we use Dynkin's formula to obtain
$$\varphi(x)\geq[\min_{z\in A} \varphi(z)]\Prob_x(\uuptau_A< \uptau_\sB)\geq \kappa_4\kappa_5\, \sM_\varphi(r), \quad \text{for all}\; x\in  \overline{B(0, 2r)}\setminus \overline{B(0, r)}.$$
This implies that 
$$\inf_{\overline{B(0, 2r)}\setminus \overline{B(0, r)}}\varphi(x)\geq \kappa_4\kappa_5\, \sM_\varphi(r),$$
This completes the proof of (b) by choosing $\kappa_3^{-1}=1\wedge (\kappa_4\kappa_5)$
since the value $\sM_\varphi(2r)$ is attained by some point $x\in\overline{B(0, 2r) \cap B^c} $.
\end{proof}

We also need the following maximum principle for super-solutions. Since any classical super-solution is also a viscosity super-solution we state the following
result for super-solutions.
\begin{lemma}\label{L-vis}
Let $u\in\cC(\Rd)\cap L^1(\Rd, \omega_\alpha)$ be a non-negative viscosity supersolution of $\FLa u\geq 0$ in $D$, for some open set $D$. If for some $x_0\in D$
we have $u(x_0)=0$ then we must have $u=0$ in $\Rd$.
\end{lemma}

\begin{proof}
Consider a test function $\psi\in\cC(\Rd)\cap L^1(\Rd, \omega_\alpha)$ such that $\psi\leq u$ in $\Rd$, and for some small $\delta>0$ with $B(x_0, 2\delta)\Subset D$,
we have $\psi=0$ in $B(x_0, \delta)$ and $\psi=u$ in $B^c(x_0, 2\delta)$. Then by the definition of viscosity super-solution (cf. \cite{CS09}) we have
$$\FLa\psi(x_0)\geq 0,$$
which implies
$$c_{d, \alpha} \int_{x_0+z\in B^c(x_0, 2\delta)} u(x_0+z) \frac{1}{\abs{z}^{d+\alpha}} dz \leq \gFLa\psi(x_0)\leq 0,$$
for some suitable constant $c_{d, \alpha}$. Since $\delta$ can be chosen arbitrarily small we get that $u=0$ in $\Rd$. Hence the proof.
 \end{proof}

Now we are ready to prove our main results.
\begin{proof}[Proof of Theorem~\ref{T2.1}]
We complete the proof in three steps.

\noindent{\bf Step 1.} We show that if either of $u$ and $v$ vanishes at some point in $B^c$ then both of them are identically $0$. To show this, we suppose that $u(z)=0$ for some $z\in B^c$.
It then follows from Lemma~\ref{L-vis} that $u=0$ in $\Rd$. Pick $\delta>0$ so that $B(z, 2\delta)\subset B^c$ and let $\uptau_\delta$ be the exit time of $X$
 from this ball. Then by 
\eqref{Dynkin1} we have
\begin{align}\label{ET2.1B}
0=u(z)\geq \Exp_z\left[\int_0^{\uptau_\delta}\FLa u(X_s) ds\right]& \geq \Exp_z\left[\int_0^{\uptau_\delta}f(X_s, v(X_s)) ds\right]\nonumber
\\
&=\int_{B(z, 2\delta)} G_{\alpha, B(z, 2\delta)}(z, y) f(y, v(y)) dy.
\end{align}
By Assumption~\ref{Ass-1} it then follows that $v=0$ in $B(z, 2\delta)$. Again, Lemma~\ref{L-vis} gives us $v=0$ in $\Rd$.

\noindent{\bf Step 2.} We suppose that $u>0$ and $v>0$ in $B^c$. We claim that
\begin{equation}\label{ET2.1C}
\inf_{B^c} u=0=\inf_{B^c}v.
\end{equation}
We may assume that $u>0, v>0$ on $\partial B$. Otherwise, we enlarge $B$.
To prove it by contradiction, we assume that $\inf_{B^c} u=c>0$. Pick $\abs{x}=n+2$ and let $\upsigma_n$ be the exit time from the ball $B(x, n)$. Using the second equation of \eqref{A1} and  
 \eqref{Dynkin1} we have
$$v(x)\geq \Exp_z\left[\int_0^{\upsigma_n}g(Y_s, u(Y_s)) ds\right]=\int_{B(x, n)} G_{\beta, B(x, n)}(x, y) g(y, u(y)) dy.$$
By the monotonicity property  of $g$ and Assumption~\ref{Ass-1} we get that for $y\in B^c$
$$ g(y, u(y))\geq g(y, c)\geq \kappa V(y),$$
for some constant $\kappa>0$. Hence using Lemma~\ref{L2.1}, we have for some constant $\kappa_1$ that
$$v(x)\geq \kappa_1 \frac{1}{n^{d-\beta}}\int_{B(x, \frac{n+2}{4})} V(y) dy\geq \kappa_1 \frac{1}{n^{d-\beta}} \Phi_V(n) \to \infty,$$
as $n\to\infty$, where the last line usage \eqref{ET2.1A}. Thus we have $\lim_{\abs{x}\to\infty} v(x)=\infty$. In particular, $\inf_{B^c}v>0$. Therefore, again using \eqref{A1}
and repeating the above argument, we find $\lim_{\abs{x}\to\infty} u(x)=\infty$.
Since $\alpha, \beta<d$ we have both the processes transient, i.e. $\lim_{t\to\infty} |X_t|=\infty$ almost surely. It also implies that for some $x\in B^c(0, 2)$ we have $\Prob_x(\uuptau_{B(0, 2)}<\infty)<1$ where 
$\uuptau_{A}$ denotes the hitting time to $A$. Otherwise, if $\Prob_x(\uuptau_{B(0, 2)}<\infty)=1$ holds for all $x$ then the process will enter $B(0, 2)$ infinitely often, in particular, it will be recurrent,
and never go to infinity, violating the property of transience. Pick such a point $x$ and for large $n$ we apply Dynkin's formula \eqref{Dynkin1} to obtain
$$u(x)\geq \Exp_x[u(X_{\uptau_n\wedge\uuptau_{B(0, 2)}})]\geq [\inf_{\abs{z}\geq n}u(z)] \Prob_x(\uuptau_{B(0, 2)}>\uptau_n)\geq [\inf_{\abs{z}\geq n}u(z)] \Prob_x(\uuptau_{B(0, 2)}=\infty),$$
where $\uptau_n$ denotes the exit time of $X$ from $B(0, n)$.
Letting $n\to\infty$, we see that $u(x)=\infty$ which is a contradiction. This gives us \eqref{ET2.1C}.

\noindent{\bf Step 3.} We assume that $u>0$ and $v>0$ in $B^c$ and arrive at a contradiction in this step. Assume that \eqref{ET2.1A1} holds.
We define a sequence $r_n\to \infty$ as follows: set $r_1=2$ and defne
$$r_n=\inf\{r>0 \; :\; \sM_u(r)\leq \frac{1}{2}\sM_u(r_{n-1})\},$$
where $\sM_u$ is given by \eqref{defi}.
It is evident from \eqref{ET2.1C} that $r_n\to\infty$. Moreover, we can find $x_n$ with $\abs{x_n}=r_n$ such that $u(x_n)=\sM_u(r_n)$. We claim that
\begin{equation}\label{ET2.1D}
\delta_v(n) :=\inf_{B(x_n, \frac{r_n}{4})} v(y)\to 0, \quad \text{as}\quad n\to\infty.
\end{equation}
If not, then we can find a subsequence, say $\{x_{n_k}\}$, satisfying $\lim_{n_k\to\infty}\delta_v(n_k)=\delta>0$. Applying Dynkin's formula to \eqref{A1} we note that
for $\uptau_{n_k}=\uptau_{B(x_{n_k}, r_{n_k}-1)}$
\begin{align*}
u(x_{n_k})\geq \Exp_{x_{n_k}}\left[\int_0^{\uptau_{n_k}} f(X_s, v(X_s)) ds\right]& \geq \int_{B(x_{n_k}, \frac{r_{n_k}}{4})} G_{\alpha, B(x_{n_k}, r_{n_k}-1)}(x_{n_k}, y) f(y, \frac{\delta}{2})dy
\\
&\geq \kappa \frac{1}{r_{n_k}^{d-\alpha}} \int_{B(x_{n_k}, \frac{r_{n_k}}{4})} U(y) dy
\\
&\geq \kappa \frac{1}{r_{n_k}^{d-\alpha}} \Phi_U(r_{n_k}),
\end{align*}
for some constant $\kappa$, where we use Lemma~\ref{L2.1} in the second line.
 Now by \eqref{ET2.1A} the right hand side converges to infinite as $r_{n_k}\to\infty$. But this is a contradiction as $u(x_{n_k})\to 0$. Therefore, we have \eqref{ET2.1D}.
Now using \eqref{ET2.1D}, \eqref{A2.1} and repeating the same calculation as above we arrive at
\begin{equation}\label{ET2.1E}
\sM_u(r_n)\geq \kappa [\delta_v(n)]^p \frac{1}{r_{n}^{d-\alpha}} \Phi_U(r_{n}).
\end{equation}
Pick any $z\in B(x_n, \frac{r_n}{4})$, and apply \eqref{Dynkin1} and Lemma~\ref{L2.1} to get
\begin{align*}
v(z)\geq \Exp_{z}\left[\int_0^{\upsigma_{B(z, \frac{r_n}{2})}} g(Y_s, u(Y_s)) ds\right] &=\int_{B(z, \frac{r_n}{2})} G_{\beta, B(z, \frac{r_n}{2})}(z, y) g(y, u(y)) dy
\\
&\geq \kappa_1 [\sM_u(2r_n)]^q \int_{B(z, \frac{r_n}{4})} G_{\beta, B(z, \frac{r_n}{2})}(z, y) V(y) dy
\\
&\geq \kappa_2 [\sM_u(2r_n)]^q \frac{1}{r_n^{d-\beta}} \int_{B(z, \frac{r_n}{4})}  V(y) dy
\\
&\geq \kappa_2 [\sM_u(2r_n)]^q \frac{1}{r_n^{d-\beta}} \Phi_V(r_n),
\end{align*}
for some constant $\kappa_1, \kappa_2$, where in the second line we use \eqref{A2.1}. This  gives us
$$ \delta_v(n)\geq \kappa_2 [\sM_u(2r_n)]^q \frac{1}{r_n^{d-\beta}} \Phi_V(r_n)\geq \kappa_3 [\sM_u(r_n)]^q \frac{1}{r_n^{d-\beta}} \Phi_V(r_n),$$
where in the last line we use Lemma~\ref{L2.2}(b). Putting this back in \eqref{ET2.1E} we find that for some constant $\kappa_4$ we have
$$\kappa_4\geq \frac{1}{r_{n}^{d-\alpha}} \Phi_U(r_{n}) \frac{1}{r_n^{p(d-\beta)}} (\Phi_V(r_n))^p (\sM_u(r_n))^{pq-1}
= \frac{1}{r_{n}^{(p+1)d-p\beta-\alpha}} \Phi_U(r_{n}) (\Phi_V(r_n))^p (\sM_u(r_n))^{pq-1}. $$
By \eqref{ET2.1A} it is easy to see that $pq\leq1$ leads to a contradiction. Indeed, since $\lim_{r\to\infty} \sM_u(r)=0$ (by \eqref{ET2.1C}), we
get from above that 
$$\kappa_4\geq \frac{1}{r_{n}^{(p+1)d-p\beta-\alpha}} \Phi_U(r_{n}) (\Phi_V(r_n))^p= \left[\frac{\Phi_U(r_n)}{r_n^{d-\alpha}}\right]^p\cdot \left[\frac{\Phi_V(r_n)}{r_n^{d-\beta}}\right]\to \infty,$$
by \eqref{ET2.1A}. This is impossible.
 So we consider the case $pq>1$.
Since $\sM(r_n)\to 0$, using Lemma~\ref{L2.2}(a) we arrive at
$$ \kappa_4 \geq \frac{1}{r_{n}^{(p+1)d-p\beta-\alpha}} \Phi_U(r_{n}) (\Phi_V(r_n))^p\frac{1}{r_n^{(d-\alpha)(pq-1)}}
= \left[\frac{1}{r_{n}^{(q+1 )d-\beta-\alpha q}} \Phi_U(r_{n})^{\frac{1}{p}} (\Phi_V(r_n))\right]^p .$$
But this is a contradiction to \eqref{ET2.1A1}. Similarly, if \eqref{ET2.1A2} holds, then we start with $v$ and get a contradiction.

Thus the only possible solution to \eqref{A1} is $u=0=v$.

For the second part we may assume that $U=V=c_1$. Then we get $\Phi_U(r)\simeq r^d\simeq \Phi_V(r)$. Thus \eqref{ET2.1A} holds. On the other hand, if $pq\leq 1$ we have $dp\leq \nicefrac{d}{q}$ and $dq\leq \nicefrac{d}{p}$. Therefore,
both \eqref{ET2.1A1} and \eqref{ET2.1A2} holds. Hence the proof follows from the first part.
\end{proof}
Before we proceed to prove Theorem~\ref{T2.2}, let us point out the following result that we get from the proof of Theorem~\ref{T2.1}.
\begin{theorem}\label{T2.5}
Let $u$ be a non-negative solution of 
$$\FLa u\geq f(x, u)\quad \text{in}\; B^c,$$
where $f$ satisfies \eqref{A2.1} for some $p\geq 1$. Furthermore, we assume that
$\alpha\in (0, 2\wedge d)$, and
\begin{equation}\label{ET2.5A}
\lim_{r\to\infty}\;  \frac{r^{d-\alpha}}{\Phi^{\frac{1}{p}}_U(r)}=0.
\end{equation}
Then it holds that $u=0$.
\end{theorem}

\begin{proof}
Note that \eqref{ET2.5A} also implies 
$$\lim_{r\to\infty}\;  \frac{r^{d-\alpha}}{\Phi_U(r)}=0.$$
By Lemma~\ref{L-vis}, if  $u$ vanishes at some point in $B^c$ then it is identically $0$. So we restrict
ourselves to positive super-solutions. Without loss of generality, we may assume that $u>0$ on $\partial B$, otherwise we consider a bigger ball $B$.
We can now repeat the arguments of Step 2 in the proof of Theorem~\ref{T2.1} to conclude that
\begin{equation}\label{ET2.5B}
\inf_{B^c} u =0.
\end{equation}
Now choose a sequence $\{(r_n, x_n)\}$ as in Step 3 of Theorem~\ref{T2.1} i.e.,
$r_1=2$, and
$$r_n=\inf\{r>0 \; :\; \sM_u(r)\leq \frac{1}{2}\sM_u(r_{n-1})\},$$
and $\abs{x_n}=r_n$ with $u(x_n)=\sM_u(r_n)$ where $\sM_u$ is given by \eqref{defi}. It is evident from \eqref{ET2.5B} that $r_n\to\infty$.
Denote by $\uptau_{n}=\uptau_{B(x_{n}, r_n-1)}$. Using \eqref{Dynkin1} and Lemma~\ref{L2.1} we then have
\begin{align*}
\sM_u(r_n)=u(x_{n}) &\geq \Exp_{x_{n_k}}\left[\int_0^{\uptau_{n}} f(X_s, u(X_s)) ds\right]
\\
& \geq \int_{B(x_{n}, \frac{r_n}{4})} G_{\alpha, B(x_{n}, r_n-1)}(x_{n}, y) f(y, \sM_u(2r_n))dy
\\
&\geq \kappa (\sM_u(2r_n))^p\frac{1}{r_{n}^{d-\alpha}} \int_{B(x_{n}, \frac{r_n}{4})} U(y) dy
\\
&\geq \kappa (\sM_u(r_n))^p \frac{1}{r_{n_k}^{d-\alpha}} \Phi_U(r_{n}),
\end{align*}
for some constant $\kappa$, where in the last line we use Lemma~\ref{L2.2}(b).
Thus for some constant $\kappa_1$ we have
$$\frac{1}{\kappa}\geq \sM_u^{p-1}(r_n) \frac{1}{r_n^{d-\alpha}} \Phi_U(r_n)\geq \kappa_1\,\Phi_U(r_n) \frac{1}{r_n^{p(d-\alpha)}},$$
by Lemma~\ref{L2.2}. But this contradicts \eqref{ET2.5A}. 
Hence the proof.
\end{proof}

Next we prove Theorem~\ref{T2.2}.
\begin{proof}[Proof of Theorem~\ref{T2.2}]
The proof is very similar to the proof of Theorem~\ref{T2.1}. First we show that both $u$ and $v$ both can not be positive in $\Rd$.
Suppose to the contrary that $u>0$ and $v>0$ in $\Rd$.
By \eqref{A2.2} we see that for some $\delta>0$ we have
\begin{equation}\label{ET2.2A0}
f(x, t, s)\geq f(x, t\wedge \delta, s) \geq f(x, t\wedge \delta, s\wedge \delta)\geq \frac{1}{2} (t\wedge \delta)^{p_1} (s\wedge \delta)^{p_2} U(x).
\end{equation}
Similarly, we also have
$$g(x, t, s)\geq \frac{1}{2} (t\wedge \delta)^{q_1} (s\wedge \delta)^{q_2} V(x).$$
We claim that 
\begin{equation}\label{ET2.2B}
\inf_{\Rd} u=0=\inf_{\Rd}v.
\end{equation}
If not, suppose $\inf_{\Rd} u=c>0$. Then observe that
$$\FLb v\geq g(x, c, v),$$
where $g(x, c, v)$ satisfies the conditions of Theorem~\ref{T2.5}. Therefore, it must holds that $v=0$ which is a contradiction. This gives us \eqref{ET2.2B}.
Now pick any $x$ from $B(0, r)$ for some large $r$. Applying \eqref{Dynkin1} in the first equation of \eqref{B1} we obtain
\begin{align*}
u(x) &\geq \Exp_x\left[\int_0^{\uptau_{B(x, r)}}f(X_s, u(X_s), v(X_s)) ds\right] 
\\
&=\int_{B(x, r)} G_{\alpha, B(x, r)}(x, y) f(y, u(y), v(y))dy
\\
&\geq\int_{B(x, r)} G_{\alpha, B(x, r)}(x, y) f(y, \sM_u(2r), \sM_v(2r))dy
\\
&\geq \kappa \int_{B(x, r)}G_{\alpha, B(x, r)}(x, y) (\sM_u(2r))^{p_1} (\sM_v(2r))^{p_2} U(y) dy
\end{align*}
for some constant $\kappa$ where in the last line we use \eqref{ET2.2A0}. Now there exists $z\in B(x, r)$ such that $r/2<|z|<r$, $B(z, r/4)\subset B(x, r)$ and $\dist(B(z, r/4), B^c(x, r))>\frac{r}{8}$.
Thus, using Lemma~\ref{L2.1} and ~\ref{L2.2}(b), we see that for some constant $\kappa_1$ we have for any $x\in B(0, r)$
$$u(x)\geq \frac{\kappa_1}{r^{d-\alpha}} (\sM_u(r))^{p_1} (\sM_v(r))^{p_2} \Phi_U(r),$$
which, in turn, gives
\begin{equation}\label{ET2.2C}
\frac{1}{\kappa_1}\geq \frac{1}{r^{d-\alpha}}(\sM_u(r))^{p_1-1} (\sM_v(r))^{p_2} \Phi_U(r).
\end{equation}
Again, starting from the second equation in \eqref{B1} and repeating a similar calculation we arrive at
\begin{equation}\label{ET2.2D}
\frac{1}{\kappa_2}\geq \frac{1}{r^{d-\beta}} (\sM_u(r))^{q_1} (\sM_v(r))^{q_2-1} \Phi_V(r),
\end{equation}
for some constant $\kappa_2$. Multiplying \eqref{ET2.2C} and \eqref{ET2.2D} we obtain
\begin{align*}
\frac{1}{\kappa_1\kappa_2} &\geq \frac{1}{r^{2d-\alpha-\beta}} (\sM_u(r))^{p_1+q_1-1} (\sM_v(r))^{p_2+q_2-1} \Phi_U(r)\Phi_V(r)
\\
&\geq \kappa_3 \frac{1}{r^{2d-\alpha-\beta}} \frac{1}{r^{(d-\alpha)(p_1+q_1-1)}}\frac{1}{r^{(d-\beta)(p_2+q_2-1)}}\Phi_U(r)\Phi_V(r),
\\
&\geq \kappa_3 r^{\alpha(p_1+q_1)+\beta(p_2+q_2)} \frac{1}{r^{d(p_1+q_1+p_2+q_2)}}\Phi_U(r)\Phi_V(r),
\end{align*}
for some constant $\kappa_3$. But this is a contradiction to \eqref{ET2.2A}.

Therefore, either $u$ or $v$ must vanish somewhere in $\Rd$. Then the proof follows from Lemma~\ref{L-vis}.
\end{proof}

Now we proceed to prove Theorem~\ref{T2.3}
\begin{proof}[Proof of Theorem~\ref{T2.3}]
$u, v$ being non-negative super-solutions we note from Lemma~\ref{L-vis} that if $u$ (or $v$) vanishes somewhere in $\Rd$ then it is identically $0$.
Therefore, to complete the proof of Theorem~\ref{T2.3} we only need to show that $u, v$ both can not be positive in $\Rd$. We suppose to the contrary that $u, v$ are positive in
$\Rd$ and then arrive at a contradiction in each of the cases (i)--(iv). For this proof we define
$$\sM_u(r)=\inf_{x\in B(0, r)} u(x), \quad \text{and}\quad \sM_v(r)= \inf_{x\in B(0, r)} v(x).$$
Note that Lemma~\ref{L2.2}(a) holds for $\sM_u$ and $\sM_v$ and the proof of Lemma~\ref{L2.2} shows that Lemma~\ref{L2.2}(b) also holds in this case.
Consider any $x\in B(0, r)$ and apply \eqref{Dynkin1} to obtain (see the proof of Theorem~\ref{T2.2})
\begin{align*}
u(x) &\geq \Exp_x\left[\int_0^{\uptau_{B(x, r)}}U(X_s) u^{p_1}(X_s)v^{p_2}(X_s) ds\right]
\\
&\geq (\sM_u(2r))^{p_1} (\sM_v(2r))^{p_2} \int_{B(x, r)}G_{\alpha, B(x, r)}(x, y) U(y) dy.
\end{align*}
Thus, applying Lemma~\ref{L2.1} and ~\ref{L2.2}(b) we get \eqref{ET2.2C} (This is again similar to the proof in Theorem~\ref{T2.2}). Similarly, we also have \eqref{ET2.2D}. 

Now we consider (i). Multiplying \eqref{ET2.2C} and \eqref{ET2.2D} and using Lemma~\ref{L2.2}(a) we
get that
$$r^{\alpha(p_1+q_1)+\beta(p_2+q_2)} \frac{1}{r^{d(p_1+q_1+p_2+q_2)}}\Phi_U(r)\Phi_V(r)\leq \kappa,$$
for some constant $\kappa$. But this is a contradiction to \eqref{ET2.3A}. Thus either $u$ or $v$ must vanish in $\Rd$.

Now suppose (ii) holds. Then from \eqref{ET2.2D} we see that for all $r$ large we have
\begin{equation*}
(\sM_v(r))^{1-q_2} \geq \kappa_2 \frac{1}{r^{d-\beta}} (\sM_u(r))^{q_1}  \Phi_V(r).
\end{equation*}
Putting this estimate in \eqref{ET2.2C},  we have
\begin{align*}
\frac{1}{\kappa_1^{1-q_2}} &\geq \frac{1}{r^{(d-\alpha)(1-q_2)}}(\sM_u(r))^{(p_1-1)(1-q_2)} (\sM_v(r))^{p_2(1-q_2)} \Phi^{1-q_2}_U(r)
\\
& \geq \kappa^{p_2}_2 \frac{1}{r^{(d-\alpha)(1-q_2)}} (\sM_u(r))^{(p_1-1)(1-q_2)}\left[\frac{1}{r^{d-\beta}}(\sM_u(r))^{q_1}  \Phi_V(r)\right]^{p_2}  \Phi^{1-q_2}_U(r)
\\
&= \kappa^{p_2}_2 \frac{1}{r^{(d-\alpha)(1-q_2)}}\frac{1}{r^{(d-\beta)p_2}}(\sM_u(r))^{(p_1-1)(1-q_2)+p_2q_1} \Phi^{p_2}_V(r)\Phi^{1-q_2}_U(r)
\end{align*}
Now if (ii)(a) holds i.e. $(p_1-1)(1-q_2)+p_2q_1<0$, then we get from above that for some constant $\kappa_3$, 
$$\frac{1}{r^{(d-\alpha)(1-q_2)}}\frac{1}{r^{(d-\beta)p_2}}(u(0))^{(p_1-1)(1-q_2)+p_2q_1} \Phi^{p_2}_V(r)\Phi^{1-q_2}_U(r)\leq \kappa_3,$$
for all $r$ large, and this would contradict \eqref{ET2.3A0}. So we consider option (ii)(b). Using Lemma~\ref{L2.2}(a) we again get
$$\frac{1}{r^{(d-\beta)p_2}} \Phi^{p_2}_V(r) \Phi^{1-q_2}_U(r)\frac{1}{r^{(d-\alpha)((1-q_2)p_1+p_2q_1)}}\leq \kappa_3$$
for some constant $\kappa_3$.
But this contradicts \eqref{ET2.3B}. Therefore, either $u$ or $v$ must vanish somewhere in $\Rd$. 

The proof of (iii) would be analogous to (ii) i.e. we start with \eqref{ET2.2C} to get a lower bound on $\sM_u(r)$ and substitute the value in \eqref{ET2.2D} to arrive at a contradiction to \eqref{ET2.3C}.

Finally, we suppose that (iv) holds. We see from \eqref{ET2.2C} and \eqref{ET2.2D} that
$$\frac{\Phi_U(r)}{r^{d-\alpha}} \frac{\Phi_V(r)}{r^{d-\beta}}\leq \frac{1}{\kappa_1\kappa_2} (\sM_u(r))^{1-p_1-q_1}(\sM_v(r))^{1-p_2-q_2}\leq \frac{1}{\kappa_1\kappa_2} (u(0))^{1-p_1-q_1}(v(0))^{1-p_2-q_2},$$
which contradicts \eqref{ET2.3C1}. Therefore, either $u$ or $v$ must vanish somewhere in $\Rd$.
\end{proof}

Finally, we prove Theorem~\ref{T2.4}.
\begin{proof}[Proof of Theorem~\ref{T2.4}]
We start by considering both $u$ and $v$ are positive in $B^c$ and then we arrive at a contradiction. First we show that
\begin{equation}\label{ET2.4B}
\mbox{either}\; \inf_{B^c} u=0,\; \mbox{or}\; \inf_{B^c}v=0.
\end{equation}
If not, then we must have $u, v\geq c>0$ in $B^c$, and therefore, by \eqref{D1}
\begin{equation}\label{ET2.4C}
\left\{
\begin{split}
\FLa u &\geq c^{p_1+p_2} U(x) \quad \text{in}\; B^c,
\\
\FLb v & \geq c^{q_1+q_2} V(x) \quad \text{in}\; B^c,
\\
u, v &\geq 0\quad \text{in}\; \Rd.
\end{split}
\right.
\end{equation}
Now using \eqref{ET2.4A}, \eqref{ET2.4C} and the arguments in Step 2 of Theorem~\ref{T2.1} we see that $u(x), v(x)\to\infty$ as $\abs{x}\to \infty$. 
Note that $u, v$ are super-solutions. Therefore, mimicking the arguments of Step 2 of Theorem~\ref{T2.1}, we arrive at a contradiction.
Hence we must have \eqref{ET2.4B}. 

With no loss of generality, we assume that $\inf_{B^c} u=0$. Define $\sM_u$ and $\sM_v$ as in \eqref{defi}.
We choose a sequnece $\{(r_n, x_n)\}$ as in Step 3 of Theorem~\ref{T2.1} i.e.,
$\abs{x_n}=r_n$, $u(x_n)=\sM_u(r_n)$ and $r_n\to\infty$ as $n\to\infty$. Define
$$\delta_v(n)=\inf_{B(x_n, \frac{r_n}{4})} v(y).$$
Let $\uptau_{n}=\uptau_{B(x_{n}, r_n-1)}$. Applying \eqref{Dynkin1} to the first equation of \eqref{D1} we see that
\begin{align*}
u(x_{n}) &\geq \Exp_{x_{n}}\left[\int_0^{\uptau_{n}} U(X_s) u^{p_1}(X_s)v^{p_2}(X_s) ds\right] 
\\
& \geq \int_{B(x_{n}, \frac{r_n}{4})} G_{\alpha, B(x_{n}, r_n-1)}(x_{n}, y) U(y) u^{p_1}(y) v^{p_2}(y) dy
\\
&\geq \kappa \frac{1}{r_{n}^{d-\alpha}} \int_{B(x_{n}, \frac{r_n}{4})} U(y) u^{p_1}(y) v^{p_2}(y) dy
\\
&\geq \kappa \frac{1}{r_{n}^{d-\alpha}} (\sM_u(2r_n))^{p_1}  [\delta_v(n)]^{p_2} \Phi_U(r_{n_k}),
\end{align*}
for some constant $\kappa$, where we use Lemma~\ref{L2.1} in the third line. Thus, using Lemma~\ref{L2.2}(b) we have
\begin{equation}\label{ET2.4D}
\kappa_1\geq (\sM_u(r_n))^{p_1-1} [\delta_v(n)]^{p_2}\frac{\Phi_U(r_n)}{r_n^{d-\alpha}},
\end{equation}
for some constant $\kappa_1$. Let $z_n\in \overline{B(x_n, \frac{r_n}{4})}$ be such that $v(z_n)= \delta_v(n)$.
Apply \eqref{Dynkin1} to the second equation in \eqref{D1} to note
\begin{align*}
\delta_v(n)=v(z_n) & \geq \Exp_{z_{n}}\left[\int_0^{\uptau_{B(z_n, \frac{r_n}{2})}} V(X_s) u^{q_1}(X_s)v^{q_2}(X_s) ds\right]
\\
&\geq \int_{B(z_n, \frac{r_n}{2})} G_{\beta, B(z_{n}, \frac{r_n}{2})}(z_{n}, y) V(y) u^{q_1}(y) v^{q_2}(y) dy
\\
&\geq  [\sM_u(2r_n)]^{q_1} [\sM_v(2r_n)]^{q_2} \int_{B(z_n, \frac{r_n}{4})} G_{\beta, B(z_{n}, \frac{r_n}{2})}(z_{n}, y) V(y)  dy
\\
&\geq \kappa_2 [\sM_u(2r_n)]^{q_1} [\sM_v(2r_n)]^{q_2} \frac{1}{r_n^{d-\beta}} \Phi_V(r_n)
\\
&\geq \kappa_3 [\sM_u(r_n)]^{q_1} [\sM_v(r_n)]^{q_2}  \frac{1}{r_n^{d-\beta}} \Phi_V(r_n)
\\
&\geq \kappa_4 [\sM_u(r_n)]^{q_1}  \frac{1}{r_n^{(d-\beta)(q_2+1)}} \Phi_V(r_n),
\end{align*}
for some constants $\kappa_2, \kappa_3, \kappa_4$, where in the fourth line we use Lemma~\ref{L2.1}, Lemma~\ref{L2.2}(b) in the fifth line and
in the sixth line Lemma~\ref{L2.2}(a). Now using \eqref{ET2.4D} we obtain for some constant $\kappa$ that
\begin{align}\label{ET2.4E}
\kappa\geq (\sM_u(r_n))^{p_2 q_1+p_1-1}  \frac{1}{r_n^{(d-\beta)(p_2q_2+p_2)}} \Phi^{p_2}_V(r_n) \frac{\Phi_U(r_n)}{r_n^{d-\alpha}}.
\end{align}
Now if $p_2 q_1+p_1-1< 0$, then \eqref{ET2.4E} gives us
$$ \left[ \frac{\Phi_V(r_n)}{r_n^{(d-\beta)(1+q_2)}}\right]^{p_2} \frac{\Phi_U(r_n)}{r_n^{d-\alpha}}\leq \kappa (\sM_u(r_n))^{1-p_2 q_1-p_1}\xrightarrow{n\to\infty} 0,$$
and this would be contradicting \eqref{ET2.4A11}. So we consider  $p_2 q_1+p_1-1> 0$. Then Lemma~\ref{L2.2}(a) gives us
\begin{align*}
\kappa_5 &\geq \frac{1}{r_n^{(d-\alpha)(p_2 q_1+p_1-1)}}\frac{1}{r_n^{(d-\beta)(p_2q_2+p_2)}} \Phi^{p_2}_V(r_n) \frac{\Phi_U(r_n)}{r_n^{d-\alpha}}
\\
& \geq \frac{1}{r_n^{(d-\alpha)(p_2 q_1+p_1)}}\frac{1}{r_n^{(d-\beta)(p_2q_2+p_2)}} \Phi^{p_2}_V(r_n) \Phi_U(r_n),
\end{align*}
for some constant $\kappa_5$. But this is also not possible due to \eqref{ET2.4A1}. Similarly, we would also get a contradiction if 
$\inf_{B^c}v=0$.

Thus it must hold that either $u$ or $v$ should vanish somewhere in $B^c$. The proof now follows from Lemma~\ref{L-vis}.
\end{proof}

\subsection{Extension to viscosity super-solutions}
In this section we extend our previous results to continuous viscosity super-solutions.
Let $\varphi\in \cC(\Rd)$ be a viscosity super-solution to
\begin{equation}\label{E2.4}
\FLa\varphi\geq F(x)\quad \text{in} \; D, \quad \text{and}\quad \varphi=H(x) \quad \text{in}\; D^c.
\end{equation}
Here $D$ is a bounded domain with $\cC^{1,1}$ boundary and $F, H$ are continuous functions. We also assume that $H$ is bounded from below.
It is then easy to check that $\varphi$ is a super-solution with boundary data $H_n=H\wedge n$ for all large $n$.
We define for $x\in D$
\begin{align*}
w_n(x) &=\Exp_x\left[H_n(X_{\uptau_D})\right] + \Exp_x\left[\int_{0}^{\uptau_D} F(X_s) ds\right]
\\
&= \int_{D^c} H_n(y) K_D(x, y) dy + \int_{D} G_D(x, y) dy,
\end{align*}
where $K_D$ denote the Poisson kernel and $G_D$ denotes the Green function in $D$. It is known that $G_D$ is continuous in $D\times D$ outside the diagonal \cite{CS05}
and $K_D$ is continuous in $D\times (D^c)^0$ \cite[Lemma~6]{B97}.
 Therefore, it follows that $w_n$ is continuous in $D$. Again, since $D$ has $\cC^{1,1}$ boundary, all its
boundary points are regular i.e. $\Prob_z(\uptau_D=0)=\Prob_z(\uptau_{\bar D}=0)=1$ for all $z\in \partial D$ (cf.\ \cite[Lemma~2.9]{BGR}). Thus given a sequence $D\ni x_m\to x\in \partial D$ we have
$X_{\uptau_D}(x_m)\to x$ and $\uptau_D(x_m)\to 0$ in probability as $m\to\infty$, where $X(x_m)$ denotes the process staring from $x_m$ and
$\uptau_D(x_m)$ denotes the exit time corresponding to this process. Since $F$ is bounded in $D$ and $H_n$ is bounded and continuous we get that
$w_n(x_m)\to H_n(x)$ as $m\to\infty$. Thus $w_n\in \cC(\Rd)$ with $w_n=H_n$ in $D^c$. Let $t\geq 0$. Then using the strong Markov property (cf.\ \cite[p.~36]{Pro-book}) we find that
\begin{align}\label{E2.5}
&\Exp_x\left[w_n(X_{t\wedge \uptau_D})\right] + \Exp_x\left[\int_{0}^{t\wedge \uptau_D} F(X_s) ds\right]\nonumber
\\
&= \Exp_x\left[\Ind_{\{t\geq \uptau_D\}}H_n(X_{\uptau_D})\right] + \Exp_x\left[\Ind_{\{t< \uptau_D\}}w_n(X_{t})\right]\nonumber
\\
& \quad + \Exp_x\left[\Ind_{\{t\geq \uptau_D\}} \int_{0}^{\uptau_D} F(X_s) ds\right] + \Exp_x\left[\Ind_{\{t< \uptau_D\}} \int_{0}^{t} F(X_s) ds\right]\nonumber
\\
&= \Exp_x\left[\Ind_{\{t\geq \uptau_D\}}H_n(X_{\uptau_D})\right] + \Exp_x\left[\Ind_{\{t< \uptau_D\}}E_{X_t}[H_n(X_{\uptau_D})]\right]
+ \Exp_x\left[\Ind_{\{t< \uptau_D\}}E_{X_t}\left[\int_0^{\uptau_D} F(X_s) ds\right]\right]\nonumber
\\
& \quad + \Exp_x\left[\Ind_{\{t\geq \uptau_D\}} \int_{0}^{\uptau_D} F(X_s) ds\right] + \Exp_x\left[\Ind_{\{t< \uptau_D\}} \int_{0}^{t} F(X_s) ds\right]\nonumber
\\
&= \Exp_x\left[\Ind_{\{t\geq \uptau_D\}}H_n(X_{\uptau_D})\right] + \Exp_x\left[\Ind_{\{t< \uptau_D\}}H_n(X_{\uptau_D})\right]
+ \Exp_x\left[\Ind_{\{t< \uptau_D\}}\int_t^{\uptau_D} F(X_s) ds\right]\nonumber
\\
& \quad + \Exp_x\left[\Ind_{\{t\geq \uptau_D\}} \int_{0}^{\uptau_D} F(X_s) ds\right] + \Exp_x\left[\Ind_{\{t< \uptau_D\}} \int_{0}^{t} F(X_s) ds\right]\nonumber
\\
&=\Exp_x\left[H_n(X_{\uptau_D})\right] + \Exp_x\left[\int_{0}^{\uptau_D} F(X_s) ds\right] = w_n(x).
\end{align}
Now pick any point $x\in D$ and a ball $B(x, \delta)\subset D$.
By $\uptau_{x, \delta}$ denote the exit time from $B(x, \delta)$. Using the strong Markov property and \eqref{E2.5} we get that
\begin{equation}\label{E2.55}
w_n(x)=\Exp_x[w_n(X_{t\wedge \uptau_{x, \delta}})] + \Exp\left[\int_{0}^{t\wedge\uptau_{x, \delta}} F(X_s) ds\right], \quad t\geq 0.
\end{equation}
With this relation in hand we can check that $w_n$ is a viscosity solution (see \cite[Remark~3.2]{BL17b})  to 
\begin{equation}\label{E2.6}
\FLa w_n = F(x)\quad \text{in} \; D, \quad \text{and}\quad \varphi=H_n(x) \quad \text{in}\; D^c.
\end{equation}
Indeed, consider a bounded test function $\psi\in \cC(\Rd)$ satisfting $\psi(y)>w_n(y)$ for all $\Rd\setminus\{x\}$ for some $x\in D$, $\psi(x)=w_n(x)$ and $\psi\in\cC^2(B(x, \delta))$ for some
$\delta$ small. In fact, we can choose $\delta$ small enough so that $B(x, \delta)\subset D$. Then applying \eqref{Dynkin} and using \eqref{E2.55} we get
\begin{align*}
\Exp\left[\int_{0}^{t\wedge\uptau_{x, \delta}} \FLa\psi(X_s) ds\right] &= \psi(x) -\Exp_x[\psi(X_{t\wedge \uptau_{x, \delta}})]
\\
&\leq w_n(x) - \Exp_x[w_n(X_{t\wedge \uptau_{x, \delta}})]
\\
&= \Exp\left[\int_{0}^{t\wedge\uptau_{x, \delta}} F(X_s) ds\right].
\end{align*}
Since $\Prob(\uptau_{x, \delta}>0)=1$, dividing the both sides above by $t$ and letting $t\to 0$ we obtain $\FLa\psi(x)\leq F(x)$. This proves that $w_n$ is a viscosity sub-solution of \eqref{E2.6}.
Similarly, we can show that $w_n$ is also a viscosity super-solution and therefore, it is a viscosity solution to \eqref{E2.6}.
Hence applying comparison principle \cite[Theorem~5.2]{CS09} to \eqref{E2.4} and \eqref{E2.6} we must have 
$$\varphi(x)\geq w_n(x)= \Exp_x\left[H_n(X_{\uptau_D})\right] + \Exp_x\left[\int_{0}^{\uptau_D} F(X_s) ds\right].$$
Let $n\to\infty$ and apply monotone convergence theorem to get
\begin{equation}\label{E2.7}
\varphi(x)\geq \Exp_x\left[H(X_{\uptau_D})\right] + \Exp_x\left[\int_{0}^{\uptau_D} F(X_s) ds\right], \quad \text{for all}\; x\in\; D.
\end{equation}
Note that \eqref{E2.7} is the only representation that we require in the proofs of Section~\ref{S-proof} to bypass the Dynkin's formula \eqref{Dynkin1}.
Thus we have the following extension.
\begin{theorem}\label{T-viscosity}
Theorem~\ref{T2.1} -- ~\ref{T2.4} hold for continuous viscosity super-solutions.
\end{theorem}

\subsection*{Acknowlodgements}
This research was supported in part by an INSPIRE faculty fellowship and DST-SERB grant EMR/2016/004810. The author is grateful to the referees 
for careful reading of the paper and several valuable suggestions.



\end{document}